\documentclass[12pt,reqno]{amsart}

\usepackage{amssymb}
\usepackage{amsthm}

\usepackage{latexsym,amsmath,amssymb}
\usepackage{amsfonts}
\usepackage{mathrsfs}
\usepackage{bm}
\usepackage{enumitem}
\usepackage[colorlinks=true, bookmarks=true,
bookmarksnumbered=true, bookmarkstype=toc, linkcolor=blue,
urlcolor=blue, citecolor=blue]{hyperref}

\theoremstyle{plain}
\newtheorem{thm}{Theorem}[section]
\newtheorem{lem}[thm]{Lemma}
\newtheorem{cor}[thm]{Corollary}
\newtheorem{prop}[thm]{Proposition}

\theoremstyle{definition}
\newtheorem{dfn}[thm]{Definition}
\newtheorem{rem}[thm]{Remark}

\makeatletter

\@addtoreset{equation}{section}
\makeatother

\newcommand{\R}{\mathbb{R}}
\renewcommand{\phi}{\varphi}

\begin{document}

\title[ABP estimates for weighted $1$-Laplacian]{
On the Aleksandrov--Bakelman--Pucci estimates for the weighted $1$-Laplacian
}

\author[S. Kitano]{Shuhei Kitano}
\address[S. Kitano]{
Waseda Research Institute for Science and Engineering\\
Waseda University, 3-4-1 Okubo, Shinjuku-ku\\
Tokyo 169-8555, Japan
}
\email{sk.koryo@moegi.waseda.jp}



\thanks{
The author was supported in part by JSPS KAKENHI Grant-in-Aid for Early-Career Scientists (JP24K16956) and Grant-in-Aid for JSPS Fellows (JP25KJ0086).
}

\begin{abstract}
We establish Aleksandrov--Bakelman--Pucci-type estimates for viscosity subsolutions of Poisson equations associated with the weighted $1$-Laplacian and, more generally, with fully nonlinear operators acting on the Hessian in directions orthogonal to the gradient. A key feature of the proof is that the quasiconcave envelope plays the role of the concave envelope in the classical ABP estimate.
\end{abstract}

\maketitle

\section{Introduction}
Let $n\geq2$ and let $\Omega\subseteq\R^n$ be a bounded domain.
In this paper, we study Aleksandrov--Bakelman--Pucci (ABP) estimates for the Poisson problem associated with the weighted $1$-Laplacian
\[
 |\nabla u|^{1/(n-1)}\Delta_1u=f(x)
 \quad\text{in }\Omega,
\]
where
\[
 \Delta_1u
 =\operatorname{div}\frac{\nabla u}{|\nabla u|}
 =\frac{1}{|\nabla u|}\operatorname{tr}\widetilde D^2u.
\]
More generally, we consider the fully nonlinear equation
\begin{equation}\label{F}
 |\nabla u|^{(2-n)/(n-1)}
 F(x,\widetilde D^2u)=f(x).
\end{equation}
Here $\widetilde D^2u$ denotes the Hessian of $u$ projected onto directions orthogonal to $\nabla u$. We assume that $F$ is uniformly elliptic with ellipticity constants $0<\lambda\leq\Lambda$. Precise definitions and notation are given in Section~\ref{s2}.

The ABP estimate originates in the classical work of Aleksandrov
\cite{alek1961,alek1963},
Bakelman~\cite{bak1961}, and Pucci~\cite{puc1966}, and is one of the fundamental tools in the regularity theory of elliptic equations in nondivergence form.
A basic version of the estimate states that if $u$ is a subsolution of
\[
 \operatorname{tr}(A(x)D^2u)=f(x)
 \quad\text{in }\Omega
\]
and $u\leq0$ on $\partial\Omega$, where $A(x)$ is uniformly elliptic, then
\[
 \sup_\Omega u
 \leq
 C\operatorname{diam}(\Omega)
 \|f^-\|_{L^n(\{\widetilde\Gamma=u^+\})},
\]
where $\widetilde\Gamma$ denotes the concave envelope of $u^+$ after its zero extension to a suitable ball containing $\Omega$.
Corresponding ABP estimates hold for fully nonlinear uniformly elliptic equations; see~\cite{CafCab}.
It also plays a fundamental role in the Krylov--Safonov theory and in the derivation of the Harnack inequality; see the foundational works~\cite{KS79,Saf80}.

ABP-type estimates have also been established for degenerate and singular equations. Important examples include the Poisson equation for the $p$-Laplacian,
\begin{equation}\label{plap}
 \Delta_pu
 =
 \operatorname{div}(|\nabla u|^{p-2}\nabla u)
 =f,
\end{equation}
for $p\in(1,\infty)$, and equations of the form
\begin{equation}\label{ds}
 |\nabla u|^\alpha F(D^2u)=f,
\end{equation}
for $\alpha>-1$. For the $p$-Laplacian, see~\cite{ACP,Miotto10,WW13}, and for singular or degenerate fully nonlinear equations, see~\cite{DFQ09,I11,Miotto10}.

Another extreme example is the Poisson equation for the normalized infinity Laplacian,
\begin{equation}\label{ilap}
 \Delta_\infty^N u
 =
 |\nabla u|^{-2}
 \langle D^2u\nabla u,\nabla u\rangle
 =f.
\end{equation}
For the normalized infinity Laplacian \eqref{ilap}, the classical ABP mechanism is substantially more delicate because of the strong degeneracy of the operator. Charro, De Philippis, Di Castro, and M\'aximo~\cite{CDDM} obtained $L^\infty$ bounds in terms of $\|f\|_{L^\infty}$ and proved that the classical ABP estimate fails in this setting. They also raised the question of whether an $L^\infty$ estimate can be obtained in terms only of $\|f\|_{L^p}$ for some $p>n$. We are not aware of a result resolving this question.

It is therefore natural to ask whether an analogous estimate can be obtained for equations of $1$-Laplace type. The variational theory of the $1$-Laplacian is naturally formulated in $BV$, using total variation and the pairing between $BV$ functions and bounded divergence-measure fields; see, for example, \cite{Anz83,ACM04,Dem04,KawSch07}. To the best of our knowledge, no ABP-type estimate has previously been established for the $1$-Laplacian.
For the unweighted equation
\[
 \Delta_1u=f,
\]
no $L^\infty$ estimate depending only on $f$ and $\Omega$ can hold, since $\lambda u$ is again a solution for every $\lambda>0$ whenever $u$ is a solution.
This motivates the weighted $1$-Laplacian considered above.

Our main result is the following ABP-type estimate.

\begin{thm}\label{thm1}
Let $\Omega\subseteq\R^n$ be a bounded domain, and let
$F\in C(\overline\Omega\times\mathcal S^n)$ be uniformly elliptic with ellipticity constants
$0<\lambda\leq\Lambda$ and satisfy $F(x,0)=0$ for every $x\in\overline\Omega$.
Suppose that $u\in\operatorname{USC}(\overline\Omega)$ is a viscosity subsolution of \eqref{F} in $\Omega$ for some $f\in C(\overline\Omega)$ and that
$u\leq0$ on $\partial\Omega$.
Then
\begin{equation}\label{ABP}
 \sup_\Omega u
 \leq
 \frac{1}{
 \lambda^{n-1}(n-1)^{n-1}
 \mathcal H^{n-1}(\partial B_1)}
 \int_{\Theta(u^+,\Omega)}
 (f^-)^{n-1}\,dx.
\end{equation}
Here
\[
 \Theta(u^+,\Omega)
 =
 \left\{
 x\in\Omega:
 \begin{array}{l}
 \text{there exists }\nu\in\partial B_1\text{ such that}\\[1mm]
 u^+(y)\leq u^+(x)
 \text{ for every }y\in\Omega
 \text{ with }(y-x)\cdot\nu>0
 \end{array}
 \right\}.
\]
\end{thm}

Since $u^+$ is upper semicontinuous, $\Theta(u^+,\Omega)$ is relatively closed in $\Omega$, and hence Borel measurable.
The definition of viscosity subsolutions is given in Section~\ref{s2}.

\begin{rem}
If $u\in C(\overline\Omega)$ and $u^+$ is extended by zero outside $\Omega$, then the contact set can equivalently be characterized in terms of its quasiconcave envelope $\Gamma$:
\[
 \Theta(u^+,\Omega)
 =
 \{x\in\Omega:\Gamma(x)=u^+(x)\}.
\]
The quasiconcave envelope is defined in Section~\ref{s2}; see Lemma~\ref{lem4} for the precise statement.
\end{rem}

The proof of Theorem~\ref{thm1} is partly inspired by the argument of
Charro, De Philippis, Di Castro, and M\'aximo~\cite{CDDM}, where geometric information on level sets is combined with an area formula and the coarea formula to control the relevant second-order derivatives.
A distinctive feature of our argument is that the quasiconcave envelope plays the role occupied by the concave envelope in the classical ABP estimate.

The rest of the paper is organized as follows.
In Section~\ref{s2}, we introduce the notation and preliminary results.
In Section~\ref{s3}, we prove the ABP estimate.

\section{Notation and preliminaries}\label{s2}
For $r>0$ and $x\in\R^n$, we denote by $B_r(x)$ the open ball of radius $r$ and center $x$, and write $B_r=B_r(0)$.
Given $\xi,\eta\in\R^n$, we denote their tensor product by
$\xi\otimes\eta=(\xi_i\eta_j)_{1\leq i,j\leq n}$.
We write $\mathcal L^n$ and $\mathcal H^{n-1}$ for Lebesgue measure and $(n-1)$-dimensional Hausdorff measure, respectively, and $\mathcal S^n$ for the space of $n\times n$ symmetric matrices.
For a function $\phi$, we use the notation $\phi^+=\max\{\phi,0\}$ and $\phi^-=\max\{-\phi,0\}$.

We say that $\phi$ is twice differentiable at $x_0$ if
\[
\phi(x)=\phi(x_0)+\langle p,x-x_0\rangle
+\frac12\langle X(x-x_0),x-x_0\rangle+o(|x-x_0|^2)
\]
as $x\to x_0$, for some $p\in\R^n$ and $X\in\mathcal S^n$; we then write $p=\nabla\phi(x_0)$ and $X=D^2\phi(x_0)$.

For $p\in\R^n\setminus\{0\}$, set
\[
P_p:=I-\frac{p}{|p|}\otimes\frac{p}{|p|}.
\]
If $\phi\in C^2$ and $\nabla\phi(x)\neq0$, we define the projected Hessian by
\[
\widetilde D^2\phi(x):=P_{\nabla\phi(x)}D^2\phi(x)P_{\nabla\phi(x)}.
\]

For ellipticity constants $0<\lambda\leq\Lambda$, the Pucci operators are
\begin{align*}
\mathcal P^+(X)
&:=\max\{\operatorname{tr}(AX):A\in\mathcal S^n,\ \lambda I\leq A\leq\Lambda I\},\\
\mathcal P^-(X)
&:=\min\{\operatorname{tr}(AX):A\in\mathcal S^n,\ \lambda I\leq A\leq\Lambda I\}.
\end{align*}
We say that $F:\Omega\times\mathcal S^n\to\R$ is uniformly elliptic if
\[
\mathcal P^-(X-Y)\leq F(x,X)-F(x,Y)\leq\mathcal P^+(X-Y)
\]
for every $x\in\Omega$ and $X,Y\in\mathcal S^n$.

For the purposes of this paper, we use only the following subsolution condition, which imposes no condition at contact points where the test function has zero gradient; see~\cite{CIL} for general background on viscosity solutions.
\begin{dfn}\label{visc}
Let $F\in C(\Omega\times\mathcal S^n)$ and $f\in C(\overline\Omega)$.
A function $u\in\operatorname{USC}(\overline\Omega)$ is a viscosity subsolution of \eqref{F} in $\Omega$ if, whenever $\phi\in C^2(\Omega)$ touches $u$ from above at $x\in\Omega$ and $\nabla\phi(x)\neq0$, one has
\[
|\nabla\phi(x)|^{(2-n)/(n-1)}
F\bigl(x,\widetilde D^2\phi(x)\bigr)\geq f(x).
\]
\end{dfn}

\begin{rem}\label{re1}
For the proof of Theorem~\ref{thm1}, we use the following reduction. Let $\widetilde u$ be the zero extension of $u^+$ to $\R^n$. Since $u\leq0$ on $\partial\Omega$, one has $\widetilde u\in\operatorname{USC}(\R^n)$, $\widetilde u\geq0$, and
\[
\operatorname{supp}\widetilde u\subseteq\overline\Omega.
\]
If an upper test function $\phi$ touches $\widetilde u$ at a point where $\widetilde u=0$, then $\phi$ has a local minimum there and hence $\nabla\phi=0$. If $\widetilde u>0$ at the contact point, then that point lies in $\Omega$ and $\phi$ is also an upper test function for the original $u$.

We also fix continuous extensions $\widetilde F\in C(\R^n\times\mathcal S^n)$ and $\widetilde f\in C(\R^n)$ which agree with $F$ and $f$ on $\overline\Omega$, respectively, with $\widetilde F(x,0)=0$ and the same ellipticity constants $\lambda,\Lambda$. Such an extension of $F$ is obtained by a standard partition-of-unity construction on $\R^n\setminus\overline\Omega$, using convex combinations of the operators $F(x,\cdot)$; the normalization and ellipticity constants are preserved under convex combinations. Then $\widetilde u$ is a viscosity subsolution of the extended equation in $\R^n$. If $\Theta$ denotes the global contact set introduced in Section~\ref{s3}, then
\[
\Theta(\widetilde u)\cap\Omega=\Theta(u^+,\Omega).
\]
After this reduction we drop the tildes. Thus, from this point on, $u$, $F$, and $f$ are understood to be defined on all of $\R^n$, with $u\in\operatorname{USC}(\R^n)$, $u\geq0$, $\operatorname{supp}u\subseteq\overline\Omega$, and with $F$ and $f$ continuous.
\end{rem}

For $r>0$, set
\[
\Omega_r:=\{x\in\R^n:\operatorname{dist}(x,\Omega)<r\}.
\]

We next record the properties of the modified sup-convolution used in the proof.
For $u\in\operatorname{USC}(\R^n)$ and $\epsilon\in(0,1)$, define
\begin{equation}\label{supc}
u^\epsilon(x)
:=\sup_{y\in\R^n,\ z\in\overline{B_\epsilon}}
\left\{u(x+z-y)-\frac{|y|^2}{2\epsilon}\right\}.
\end{equation}

\begin{prop}\label{p1}
Let $u\in\operatorname{USC}(\R^n)$ be nonnegative with $\operatorname{supp}u\subseteq\overline\Omega$. Set $M:=\sup_{\R^n}u$ and
\[
c_\epsilon:=2\epsilon+\sqrt{2\epsilon M}.
\]
Then:
\begin{enumerate}
\renewcommand{\labelenumi}{(\roman{enumi})}
\item $u^\epsilon$ is $1/\epsilon$-semiconvex, i.e.
$u^\epsilon(\cdot)+|\cdot|^2/(2\epsilon)$ is convex.
\item $u^\epsilon$ is globally Lipschitz in $\R^n$.
\item
\[
\operatorname{supp}u^\epsilon\subseteq\Omega_{c_\epsilon}.
\]
\item Suppose that $F\in C(\R^n\times\mathcal S^n)$ satisfies $F(x,0)=0$ for every $x\in\R^n$ and is uniformly elliptic with constants $\lambda,\Lambda$, that $f\in C(\R^n)$, and that $u$ is a viscosity subsolution of \eqref{F} in $\{u>0\}$. Define, for $x\in\R^n$,
\begin{align*}
F_\epsilon(x,X)&:=\sup_{q\in B_{c_\epsilon}(x)}F(q,X),\\
f_\epsilon(x)&:=\inf_{q\in B_{c_\epsilon}(x)}f(q).
\end{align*}
Then $F_\epsilon\in C(\R^n\times\mathcal S^n)$, $f_\epsilon\in C(\R^n)$, $F_\epsilon(x,0)=0$, and $F_\epsilon$ is uniformly elliptic with the same constants $\lambda,\Lambda$. Moreover, $u^\epsilon$ is a viscosity subsolution of
\[
|\nabla u^\epsilon|^{(2-n)/(n-1)}
F_\epsilon\bigl(x,\widetilde D^2u^\epsilon\bigr)\geq f_\epsilon(x)
\qquad\text{in }\{u^\epsilon>0\}.
\]
\item For every sequence $x_\epsilon\to x$,
\[
\limsup_{\epsilon\to0}u^\epsilon(x_\epsilon)\leq u(x).
\]
\item For every $x\in\R^n$,
\[
u^\epsilon(x)\longrightarrow u(x)\qquad\text{as }\epsilon\to0.
\]
\end{enumerate}
\end{prop}

\begin{proof}
Write
\[
\widetilde u_\epsilon(w):=\max_{z\in\overline{B_\epsilon}}u(w+z).
\]
Then
\[
u^\epsilon(x)=\sup_{w\in\R^n}
\left\{\widetilde u_\epsilon(w)-\frac{|x-w|^2}{2\epsilon}\right\}.
\]
Properties (i), (ii), (v), and (vi), together with the viscosity-subsolution assertion in (iv), follow from the standard sup-convolution argument; see Appendix A in~\cite{CCKS}. For (iv), the transferred contact point is of the form $q_\epsilon=x+z_\epsilon-y_\epsilon$, and the estimate below shows that $|q_\epsilon-x|<c_\epsilon$.

Suppose that $u^\epsilon(x)>0$ and choose a maximizing pair
$(y_\epsilon,z_\epsilon)\in\R^n\times\overline{B_\epsilon}$ in \eqref{supc}. Set
\[
q_\epsilon:=x+z_\epsilon-y_\epsilon.
\]
Then
\[
0<u^\epsilon(x)
=u(q_\epsilon)-\frac{|y_\epsilon|^2}{2\epsilon}
\leq M-\frac{|y_\epsilon|^2}{2\epsilon}.
\]
In particular $u(q_\epsilon)>0$, so
$q_\epsilon\in\{u>0\}\subset\operatorname{supp}u$, and
$|y_\epsilon|<\sqrt{2\epsilon M}$. Hence
\[
\operatorname{dist}(x,\operatorname{supp}u)
\leq|x-q_\epsilon|
\leq\epsilon+\sqrt{2\epsilon M}.
\]
Thus
\[
\{u^\epsilon>0\}
\subset
\{x:\operatorname{dist}(x,\operatorname{supp}u)
<\epsilon+\sqrt{2\epsilon M}\}.
\]
By (ii), $u^\epsilon$ is continuous, and therefore
\[
\operatorname{supp}u^\epsilon
\subset
\{x:\operatorname{dist}(x,\operatorname{supp}u)
\leq\epsilon+\sqrt{2\epsilon M}\}
\subseteq\Omega_{c_\epsilon},
\]
where we used $\operatorname{supp}u\subseteq\overline\Omega$. This proves (iii).

Since $F$ and $f$ are continuous and $c_\epsilon$ is fixed, the suprema and infima in (iv) agree with the corresponding maxima and minima over $\overline{B_{c_\epsilon}(x)}$. A standard compactness argument therefore gives
\[
F_\epsilon\in C(\R^n\times\mathcal S^n),
\qquad
f_\epsilon\in C(\R^n).
\]
Since $F(q,0)=0$ for every $q\in\R^n$, we have $F_\epsilon(x,0)=0$. Moreover,
\begin{align*}
F_\epsilon(x,X)-F_\epsilon(x,Y)
&\leq\sup_{q\in B_{c_\epsilon}(x)}
\bigl(F(q,X)-F(q,Y)\bigr)
\leq\mathcal P^+(X-Y),\\
F_\epsilon(x,X)-F_\epsilon(x,Y)
&\geq\inf_{q\in B_{c_\epsilon}(x)}
\bigl(F(q,X)-F(q,Y)\bigr)
\geq\mathcal P^-(X-Y),
\end{align*}
so $F_\epsilon$ has the same ellipticity constants. This completes the proof of (iv).

\end{proof}

We next recall the quasiconcave envelope and its elementary level-set properties.
\begin{dfn}\label{def1}
A function $\omega:\R^n\to\R$ is quasiconcave if $\{\omega\geq a\}$ is convex for every $a\in\R$.
For $u\in\operatorname{USC}(\R^n)$ with compact support, the quasiconcave envelope of $u^+$ is
\[
\Gamma(x):=\inf\{\omega(x):\omega\text{ is quasiconcave and }\omega\geq u^+\text{ in }\R^n\}.
\]
\end{dfn}

Here and below, $\operatorname{Conv}(A)$ denotes the convex hull of $A\subseteq\R^n$.

\begin{prop}\label{prop-qc}
Let $u\in C(\R^n)$ have compact support, and let $\Gamma$ be the quasiconcave envelope of $u^+$.
Then the following properties hold.
\begin{enumerate}
\renewcommand{\labelenumi}{(\roman{enumi})}
\item For every $a>0$,
\begin{equation}\label{qc-levels-closed}
\operatorname{Conv}(\{u^+\geq a\})=\{\Gamma\geq a\}.
\end{equation}
Moreover, for every $a\geq0$,
\begin{equation}\label{qc-levels-open}
\operatorname{Conv}(\{u^+>a\})=\{\Gamma>a\}.
\end{equation}
\item Every modulus of continuity of $u^+$ is also a modulus of continuity of $\Gamma$.
In particular, if $u$ is $L$-Lipschitz, then $\Gamma$ is $L$-Lipschitz.
\end{enumerate}
\end{prop}
\begin{proof}
The level-set characterization in (i) is standard; see, e.g.,~\cite[Section~2]{ColSal03}.
For (ii), let $\omega$ be a modulus of continuity of $u^+$.
For every $x,y\in\R^n$,
\[
\Gamma(x+y)+\omega(|y|)\geq u^+(x+y)+\omega(|y|)\geq u^+(x).
\]
The function $x\mapsto\Gamma(x+y)+\omega(|y|)$ is quasiconcave and dominates $u^+$, so the minimality of $\Gamma$ yields
\[
\Gamma(x+y)+\omega(|y|)\geq\Gamma(x).
\]
Replacing $y$ by $-y$ gives
$|\Gamma(x+y)-\Gamma(x)|\leq\omega(|y|)$.
\end{proof}

\section{ABP estimates for the weighted \texorpdfstring{$1$}{1}-Laplacian}\label{s3}
    Throughout this section we work after the reduction in Remark~\ref{re1}; in particular, all functions and all envelope and contact-set constructions are understood on $\R^n$.

    For a function $v$ on $\R^n$, we write
    \[
    \Theta(v)
    :=
    \left\{
    x\in\R^n:
    \begin{array}{l}
    \text{there exists }\nu\in\partial B_1\text{ such that}\\[1mm]
    v(y)\leq v(x)\text{ whenever }(y-x)\cdot\nu>0
    \end{array}
    \right\}.
    \]
    We first give a formal argument that highlights the main idea of the proof.
    Let $\Gamma$ be the quasiconcave envelope of $u$.
    For each $a>0$, we write $K_a=\{\Gamma\geq a\}$.
    Let $\nu(x)=-\nabla\Gamma(x)/|\nabla \Gamma(x)|$ be the outward unit normal of $K_a$ at $x$, and let $\nabla^{\partial K_a}\nu$ be the Weingarten operator.
    The second-order obstacle characterization of quasiconcave envelopes obtained by Barron--Goebel--Jensen~\cite[Corollary~6.2]{BGJ} suggests, at least formally, that
    \begin{equation}
        \begin{split}
        |\nabla\Gamma|\det\nabla^{\partial K_a}\nu&=0\quad\text{a.e. on }\{\Gamma>u\}\cap \partial K_a,\quad\text{and}\\
        0\leq |\nabla \Gamma|\det\nabla^{\partial K_a}\nu&\leq \left(\frac{f^-}{\lambda(n-1)}\right)^{n-1}\quad\text{a.e. on }\{\Gamma=u>0\}\cap\partial K_a.
        \end{split}
        \label{e32}
    \end{equation}
By the area formula for the Gauss map and the coarea formula (see~\cite{MR257325}), we formally compute
    \begin{align}
        \sup_{\R^n}\Gamma\mathcal{H}^{n-1}(\partial B_1)
        &=\int_0^{\sup_{\R^n}\Gamma}\mathcal{H}^{n-1}(\partial B_1)da\nonumber\\
        &=\int_0^{\sup_{\R^n}u}\int_{\partial K_a}\det\left(\nabla^{\partial K_a}\nu \right)d\mathcal{H}^{n-1}da\nonumber\\
        &=\int_{\cup_{a>0}\partial K_a}|\nabla\Gamma|\det\left(\nabla^{\partial K_{\Gamma(x)}}\nu \right)dx\label{e2}\\
        &\leq (\lambda(n-1))^{1-n}\int_{\{\Gamma=u\}\cap\{u>0\}}(f^-)^{n-1}dx.\nonumber
    \end{align}
    Here we used the notation $K_{\Gamma(x)}=\{y\in\R^n:\Gamma(y)\geq \Gamma(x)\}$.

    Making this argument rigorous requires the Lipschitz regularity of $\Gamma$, the $C^{1,1}$ regularity of its level sets, and a suitable weak interpretation of \eqref{e32}.
    In classical ABP proofs, the standard sup-convolution provides the regularization needed for nonsmooth solutions. Here we use the modified regularization \eqref{supc}, which also yields the level-set geometry needed below.
    To obtain the $C^{1,1}$ regularity of the regularized level sets, we adapt the supporting-hyperplane argument for concave envelopes in \cite{CafCab} to the quasiconcave setting and then use the supporting-ball characterization in \cite{LP20}.
    We relate the coincidence set for continuous regularizations to the global upper contact set in Lemma~\ref{lem4}, and establish the stability of the latter in Lemma~\ref{lem3}. The passage through $\Theta$ is important for the limiting argument: coincidence sets of quasiconcave envelopes are not stable in general under the upper-semicontinuous convergence arising here, whereas the half-space formulation of $\Theta$ is stable under the half-relaxed convergence in Proposition~\ref{p1}(v)--(vi).

    We now fix $\epsilon\in(0,1)$, let $u^{\epsilon}$ be given by \eqref{supc}, and denote its quasiconcave envelope by $\Gamma^{\epsilon}$.
    Let $K^{\epsilon}_a=\{\Gamma^{\epsilon}\geq a\}$.
    The regularity of $\partial K^{\epsilon}_a$ follows from the following two-sided supporting sphere and hyperplane condition.
    \begin{lem}\label{l1}
        For every $a>0$ and $x\in \partial K^{\epsilon}_a$, there exists a unique $\nu=\nu(x)\in \partial B_1$ such that
        \[
        \overline{B_\epsilon(x-\epsilon \nu)}\subseteq K^{\epsilon}_a\subseteq \{y\in\R^n:(y-x)\cdot \nu\leq  0\}.
        \]
    \end{lem}
    \begin{proof}
        The existence of the supporting hyperplane is an immediate consequence of the Hahn--Banach theorem applied to $K^{\epsilon}_a$ and $x\in\partial K^{\epsilon}_a$.
        Note that $\nu$ is not unique at this stage.
        We fix $\nu$ and write $H_x=\{y\in\R^n:(y-x)\cdot \nu=0\}$ in what follows.
        
        It remains to prove the existence of an interior supporting ball.
        We consider two cases: \textbf{Case 1}: $x\in \partial K^{\epsilon}_a\cap \{u^{\epsilon}=\Gamma^{\epsilon}\}$ and \textbf{Case 2}: $x\in \partial K^{\epsilon}_a\cap \{u^{\epsilon}<\Gamma^{\epsilon}\}$. 

        \medskip
\noindent\textbf{Case 1.}
        Let $x_0\in \overline{B_\epsilon(x)}$ be such that
        \[
        a=u^{\epsilon}(x)=\sup_{y\in\R^n}\left( u(x_0-y)-\frac{1}{2\epsilon}|y|^2\right).
        \]
        Then, by the definition of $u^{\epsilon}$,
        $\overline{B_\epsilon(x_0)}\subseteq \{u^{\epsilon}\geq a\}\subseteq K^{\epsilon}_a$.
        Since $x\in\overline{B_\epsilon(x_0)}$ and $x\in\partial K_a^\epsilon$, necessarily $x\in\partial B_\epsilon(x_0)$. The supporting hyperplane $H_x$ is therefore tangent to this ball at $x$, and hence $x_0=x-\epsilon\nu$.

        \medskip
\noindent\textbf{Case 2.}
        Set $E_a=\{u^\epsilon\geq a\}$, so that $K_a^\epsilon=\operatorname{Conv}(E_a)$. By Carath\'eodory's theorem,
        \[
        x=\sum_{i=1}^{n+1}\lambda_i x_i,
        \qquad x_i\in E_a,\quad \lambda_i\geq0,\quad \sum_{i=1}^{n+1}\lambda_i=1.
        \]
        Since $K_a^\epsilon$ lies in the supporting halfspace $\{(y-x)\cdot\nu\leq0\}$,
        \[
        0=\sum_{i=1}^{n+1}\lambda_i(x_i-x)\cdot\nu,
        \qquad (x_i-x)\cdot\nu\leq0.
        \]
        Hence every $x_i$ with $\lambda_i>0$ belongs to $H_x$. Discarding zero coefficients and relabeling, we may assume
        \[
        x_i\in E_a\cap H_x\subseteq\partial K_a^\epsilon\cap\{u^\epsilon=\Gamma^\epsilon\}
        \quad\text{for all }i.
        \]
        Applying Case 1 at these contact points gives
        $\overline{B_\epsilon(x_i-\epsilon\nu)}\subseteq K_a^\epsilon$. Thus, for $y\in B_\epsilon$,
        \[
        x-\epsilon\nu+y
        =\sum_{i=1}^{n+1}\lambda_i(x_i-\epsilon\nu+y)\in K_a^\epsilon.
        \]
        Since $K_a^\epsilon$ is closed,
        \[
        \overline{B_\epsilon(x-\epsilon\nu)}\subseteq K_a^\epsilon.
        \]

        Finally, the supporting normal is unique. Indeed, if $\nu'\in\partial B_1$ is another supporting normal at $x$, then its supporting hyperplane must be tangent at $x$ to the ball
        $\overline{B_\epsilon(x-\epsilon\nu)}\subseteq K_a^\epsilon$. As in Case~1, the center of this ball must therefore be $x-\epsilon\nu'$. Hence
        \[
        x-\epsilon\nu=x-\epsilon\nu',
        \]
        and thus $\nu=\nu'$.
    \end{proof}
    Lemma~\ref{l1} gives an interior supporting ball of radius $\epsilon$ at every boundary point, while the supporting halfspace yields an exterior supporting ball of the same radius. Thus $K_a^\epsilon$ satisfies the uniform two-sided supporting $\epsilon$-sphere condition. By Theorem~1 and Corollary~2 in \cite{LP20}, after scaling, $\partial K_a^\epsilon$ is a $C^{1,1}$ hypersurface and its outward unit normal $\nu$ is $1/\epsilon$-Lipschitz. Hence $\nu$ is differentiable $\mathcal H^{n-1}$-almost everywhere and $\|\nabla^{\partial K_a^\epsilon}\nu\|\leq1/\epsilon$ wherever it is differentiable. Since $K_a^\epsilon$ is convex, its Weingarten operator is nonnegative semidefinite with our sign convention. Therefore:
    \begin{cor}\label{cor1}
        $\partial K^{\epsilon}_a$ is a $C^{1,1}$ hypersurface and for $\mathcal{H}^{n-1}$-almost every $x\in \partial K^\epsilon_a$,
        \[
        0\leq \langle\nabla^{\partial K^{\epsilon}_{a}}\nu(x)\xi,\xi\rangle\leq \frac{1}{\epsilon}|\xi|^2\quad\text{for }\xi\in T_x(\partial K^\epsilon_a),
        \]
        where $T_x(\partial K^\epsilon_a)$ is the tangent space of $\partial K^\epsilon_a$ at $x\in \partial K^\epsilon_a$.
    \end{cor}
    
    Next, we show that the Weingarten operator of $\partial K_a^\epsilon$ is controlled by $f_\epsilon$ at the contact points of $u^\epsilon$ and $\Gamma^\epsilon$.
    \begin{lem}\label{l3}
        For $\mathcal{L}^1$-almost every $a>0$, we have
        \[
        0\leq |\nabla \Gamma^{\epsilon}|\det(\nabla^{\partial K^{\epsilon}_{a}}\nu)\leq  \left(\frac{f_{\epsilon}^-}{\lambda(n-1)}\right)^{n-1}\quad\mathcal{H}^{n-1}\text{-a.e. on }\{\Gamma^{\epsilon}=u^\epsilon\}\cap \partial K_a^\epsilon.
        \]
    \end{lem}
    \begin{proof}
        Let $N\subseteq\R^n$ be the set where $u^\epsilon$ is not twice differentiable or $\Gamma^\epsilon$ is not differentiable. By Aleksandrov's theorem and Rademacher's theorem, $\mathcal{L}^n(N)=0$. Moreover, Corollary~\ref{cor1} implies that the Gauss map $\nu$ is Lipschitz on $\partial K_a^\epsilon$ and hence differentiable $\mathcal H^{n-1}$-almost everywhere. Since $\Gamma^\epsilon$ is continuous, $\partial K_a^\epsilon\subseteq\{\Gamma^\epsilon=a\}$. Therefore, by the coarea formula~\cite{MR257325}, for $\mathcal{L}^1$-almost every $a>0$,
        \begin{equation}\label{H}
            \mathcal{H}^{n-1}(\partial K^\epsilon_a\cap N)=0.
        \end{equation}

        Fix such an $a$, and let $x\in\partial K_a^\epsilon\cap\{\Gamma^\epsilon=u^\epsilon\}$ be a point at which $u^\epsilon$ is twice differentiable, $\Gamma^\epsilon$ is differentiable, and the Gauss map $\nu$ is differentiable. Since $u^\epsilon\leq\Gamma^\epsilon$ and equality holds at $x$,
        \[
        p:=\nabla u^\epsilon(x)=\nabla\Gamma^\epsilon(x).
        \]
        If $p=0$, the assertion is immediate. We therefore assume $p\neq0$ and set
        \[
        e:=\frac{p}{|p|}=-\nu(x),
        \qquad T:=e^\perp.
        \]
        Set
        \[
        A:=\nabla^{\partial K_a^\epsilon}\nu(x),
        \qquad
        H:=D^2u^\epsilon(x),
        \qquad
        X:=\widetilde D^2u^\epsilon(x)=P_pHP_p.
        \]
        We regard $A$ as an operator on $T$. Since $P_p$ is the orthogonal projection onto $T$, the projected matrix $X$ vanishes on $\operatorname{span}\{e\}$ and has no mixed components between $T$ and $\operatorname{span}\{e\}$. By Corollary~\ref{cor1}, $A\geq0$.

        Since $\partial K_a^\epsilon$ is a $C^{1,1}$ hypersurface with tangent space $T$ at $x$, the $C^1$ implicit function theorem gives, in a neighborhood of $0$ in $T$,
        \[
        \partial K_a^\epsilon
        =\{x+\xi+g(\xi)e:\xi\in T\},
        \qquad
        K_a^\epsilon
        =\{x+\xi+te:t\geq g(\xi)\},
        \]
        for some $g\in C^{1,1}$ satisfying $g(0)=0$ and $Dg(0)=0$. With the outward normal oriented as above,
        \[
        \nu(x+\xi+g(\xi)e)
        =\frac{Dg(\xi)-e}{\sqrt{1+|Dg(\xi)|^2}},
        \]
        where $Dg(\xi)\in T$ is identified with its Euclidean gradient. Since the Gauss map is differentiable at $x$, it follows that $Dg$ is differentiable at $0$; hence $g$ is twice differentiable there and
        \[
        D^2g(0)=A.
        \]

        Let $\xi\in T$ be a unit vector and set
        \[
        \gamma(t):=x+t\xi+g(t\xi)e\in\partial K_a^\epsilon.
        \]
        Since $\Gamma^\epsilon(\gamma(t))=a=u^\epsilon(x)$ and $u^\epsilon\leq\Gamma^\epsilon$, we have $u^\epsilon(\gamma(t))\leq u^\epsilon(x)$. Using
        \[
        g(t\xi)=\frac{t^2}{2}\langle A\xi,\xi\rangle+o(t^2)
        \]
        and the second-order expansion of $u^\epsilon$ at $x$, we obtain
        \[
        |p|\langle A\xi,\xi\rangle+\langle X\xi,\xi\rangle\leq0.
        \]
        Thus, identifying $X$ with its restriction to $T$ when comparing it with $A$,
        \begin{equation}\label{AX}
        0\leq A\leq-|p|^{-1}X.
        \end{equation}
        Since $X$ vanishes on $\operatorname{span}\{e\}$, this also shows that $X\leq0$ as a matrix on $\R^n$, so
        \[
        \mathcal P^+(X)=\lambda\operatorname{tr}X.
        \]
        For $\rho>0$, define
        \[
        \phi_\rho(y):=u^\epsilon(x)+p\cdot(y-x)
        +\frac12\langle H(y-x),y-x\rangle+\rho|y-x|^2.
        \]
        Since $u^\epsilon$ is twice differentiable at $x$, $\phi_\rho$ touches $u^\epsilon$ from above at $x$ in a sufficiently small neighborhood. Moreover,
        \[
        \widetilde D^2\phi_\rho(x)=X+2\rho P_p.
        \]
        Hence Proposition~\ref{p1}(iv) gives
        \[
        |p|^{\frac{2-n}{n-1}}F_\epsilon(x,X+2\rho P_p)\geq f_\epsilon(x).
        \]
        By the uniform ellipticity of $F_\epsilon(x,\cdot)$, it is continuous in the matrix variable. Letting $\rho\downarrow0$, we obtain
        \[
        |p|^{\frac{2-n}{n-1}}F_\epsilon(x,X)\geq f_\epsilon(x).
        \]
        Since $F_\epsilon(x,0)=0$ and $X\leq0$,
        \[
        \lambda |p|^{\frac{2-n}{n-1}}\operatorname{tr}X
        =|p|^{\frac{2-n}{n-1}}\mathcal P^+(X)
        \geq |p|^{\frac{2-n}{n-1}}F_\epsilon(x,X)
        \geq f_\epsilon(x).
        \]
        Since the left-hand side is nonpositive, $f_\epsilon(x)\leq0$, and hence
        \[
        -\operatorname{tr}X
        \leq
        \frac{f_\epsilon^-(x)}{\lambda}|p|^{\frac{n-2}{n-1}}.
        \]
        Taking traces in \eqref{AX}, we find
        \[
        \operatorname{tr}A
        \leq
        \frac{f_\epsilon^-(x)}{\lambda}|p|^{-\frac1{n-1}}.
        \]
        Since $A\geq0$, the arithmetic--geometric mean inequality yields
        \[
        \det A
        \leq
        \left(\frac{\operatorname{tr}A}{n-1}\right)^{n-1}
        \leq
        |p|^{-1}\left(\frac{f_\epsilon^-(x)}{\lambda(n-1)}\right)^{n-1}.
        \]
        Finally, $|p|=|\nabla\Gamma^\epsilon(x)|$, which proves the desired estimate.
    \end{proof}
    We next record two facts concerning contact sets and their stability.

    \begin{lem}\label{lem4}
    Let $v\in C(\R^n)$ be nonnegative and compactly supported, and let $\Gamma_v$ be its quasiconcave envelope. Then
    \[
    \{\Gamma_v=v\}=\Theta(v).
    \]
    \end{lem}
    \begin{proof}
        Let $z\in\{\Gamma_v=v\}$ and set $a=\Gamma_v(z)=v(z)$. If $\{\Gamma_v>a\}=\varnothing$, then $a=\max_{\R^n}\Gamma_v$ and hence $v(y)\leq a=v(z)$ for every $y\in\R^n$. Thus any $\nu\in\partial B_1$ shows that $z\in\Theta(v)$. We may therefore assume that $\{\Gamma_v>a\}\neq\varnothing$. The set $\{\Gamma_v>a\}$ is open and convex and does not contain $z$. Hence a supporting-hyperplane argument gives a vector $\nu\in\partial B_1$ such that
        \[
        \{v>a\}\subseteq\{\Gamma_v>a\}
        \subseteq\{y\in\R^n:(y-z)\cdot\nu<0\}.
        \]
        Therefore
        \[
        v(y)\leq v(z)\quad\text{whenever }(y-z)\cdot\nu\geq0,
        \]
        and thus $z\in\Theta(v)$.

        Conversely, suppose that $z\in\Theta(v)$ and choose $\nu\in\partial B_1$ such that
        \[
        v(y)\leq v(z)\quad\text{whenever }(y-z)\cdot\nu>0.
        \]
        By continuity, the same inequality holds when $(y-z)\cdot\nu\geq0$. Set $a=v(z)$. Then
        \[
        \{v>a\}\subseteq\{y\in\R^n:(y-z)\cdot\nu<0\}.
        \]
        Proposition~\ref{prop-qc}(i) gives
        \[
        \{\Gamma_v>a\}=\operatorname{Conv}(\{v>a\}),
        \]
        also when $a=0$. Hence
        \[
        \{\Gamma_v>a\}\subseteq\{y\in\R^n:(y-z)\cdot\nu<0\},
        \]
        so $\Gamma_v(z)\leq a$. Since $\Gamma_v\geq v$, we conclude that $\Gamma_v(z)=a=v(z)$.
    \end{proof}

    For a sequence of sets $E_j\subseteq\R^n$, we use the Kuratowski upper limit
    \[
    \limsup_{j\to\infty}E_j
    :=\{x\in\R^n:\text{there exist }j_k\to\infty\text{ and }x_{j_k}\in E_{j_k}\text{ with }x_{j_k}\to x\}.
    \]

    \begin{lem}\label{lem3}
        Let $u\in \mathrm{USC}(\R^n)$ and $u_j\in C(\R^n)$ be nonnegative. Assume that, for every $x\in\R^n$,
        \begin{enumerate}
            \item[(I)] there exist $x_j\to x$ such that $u_j(x_j)\to u(x)$;
            \item[(II)] for every sequence $x_j\to x$,
            \[
            \limsup_{j\to\infty}u_j(x_j)\leq u(x).
            \]
        \end{enumerate}
        Then
        \[
        \limsup_{j\to\infty}\Theta(u_j)
        \subseteq \Theta(u).
        \]
    \end{lem}
    \begin{proof}
        Let $z\in\limsup_{j\to\infty}\Theta(u_j)$. Passing to a subsequence, there exist $z_j\in\Theta(u_j)$ such that $z_j\to z$.
        For each $j$, choose $\nu_j\in\partial B_1$ from the definition of $\Theta(u_j)$. Passing to a further subsequence, we may assume $\nu_j\to\nu\in\partial B_1$. By (II),
        \[
        \limsup_{j\to\infty}u_j(z_j)\leq u(z).
        \]
        Fix $y\in\R^n$ such that $(y-z)\cdot\nu>0$. By (I), there exist $y_j\to y$ such that $u_j(y_j)\to u(y)$. For all sufficiently large $j$,
        $(y_j-z_j)\cdot\nu_j>0$, and hence
        \[
        u_j(y_j)\leq u_j(z_j).
        \]
        Letting $j\to\infty$ gives $u(y)\leq u(z)$. Thus $z\in\Theta(u)$.
    \end{proof}

    We can now prove the global estimate for the reduced subsolution.

    \begin{lem}\label{lem5}
        Let $u\in\mathrm{USC}(\R^n)$ be nonnegative with $\operatorname{supp}u\subseteq\overline\Omega$. Assume that $F\in C(\R^n\times\mathcal S^n)$ satisfies $F(x,0)=0$ for every $x\in\R^n$ and is uniformly elliptic with constants $\lambda,\Lambda$, that $f\in C(\R^n)$, and that $u$ is a viscosity subsolution of \eqref{F} in $\{u>0\}$. Then
        \[
        \sup_{\R^n}u
        \leq
        \frac{1}{\lambda^{n-1}(n-1)^{n-1}\mathcal H^{n-1}(\partial B_1)}
        \int_{\Theta(u)\cap\{u>0\}}(f^-)^{n-1}\,dx.
        \]
    \end{lem}
    \begin{proof}
        Set $M=\sup_{\R^n}u$. If $M=0$, there is nothing to prove. By Proposition~\ref{p1}(iii), $u^\epsilon$ is compactly supported for every $\epsilon>0$, and all these supports lie in one fixed bounded set when $\epsilon$ is sufficiently small.

        We have $\sup_{\R^n}u^\epsilon=M$: the inequality $u^\epsilon\leq M$ is immediate from the definition, while equality follows by evaluating the supremum at a maximum point of $u$ with $y=z=0$.
        Fix $\eta\in(0,M)$. For $a\in(\eta,M)$, the normal map on $\partial K_a^\epsilon$ satisfies
        \[
        \partial B_1
        \subseteq
        \nu\bigl(\partial K_a^\epsilon\cap\{u^\epsilon=\Gamma^\epsilon\}\bigr).
        \]
        Indeed, fix $\nu_0\in\partial B_1$ and set $E_a:=\{u^\epsilon\geq a\}$. Since $E_a$ is compact and $K_a^\epsilon=\operatorname{Conv}(E_a)$, the support function satisfies
        \[
        \max_{K_a^\epsilon}x\cdot\nu_0=\max_{E_a}x\cdot\nu_0.
        \]
        Choose $x_0\in E_a$ attaining this maximum. Then $x_0\in\partial K_a^\epsilon$, and hence $\Gamma^\epsilon(x_0)=a$. Since $u^\epsilon(x_0)\geq a$ and $u^\epsilon\leq\Gamma^\epsilon$, we have $u^\epsilon(x_0)=\Gamma^\epsilon(x_0)=a$. The hyperplane orthogonal to $\nu_0$ supports $K_a^\epsilon$ at $x_0$; by uniqueness of the supporting normal from Lemma~\ref{l1}, $\nu(x_0)=\nu_0$.
        Therefore, by the area formula for the Lipschitz Gauss map (see Corollary~3.2.20 in \cite{MR257325}),
        \[
        \mathcal H^{n-1}(\partial B_1)
        \leq
        \int_{\partial K_a^\epsilon\cap\{u^\epsilon=\Gamma^\epsilon\}}
        \det(\nabla^{\partial K_a^\epsilon}\nu)\,d\mathcal H^{n-1}.
        \]

        Let $Z_\epsilon:=\{x\in\R^n:\nabla\Gamma^\epsilon(x)=0\}$. Since $\Gamma^\epsilon$ is Lipschitz, the coarea formula gives
        \begin{align*}
        \int_0^M \mathcal H^{n-1}(\partial K_a^\epsilon\cap Z_\epsilon)\,da
        &\leq \int_0^M \mathcal H^{n-1}(\{\Gamma^\epsilon=a\}\cap Z_\epsilon)\,da\\
        &=\int_{Z_\epsilon}|\nabla\Gamma^\epsilon|\,dx=0.
        \end{align*}
        Hence, for almost every $a\in(\eta,M)$, the set $\partial K_a^\epsilon\cap Z_\epsilon$ is $\mathcal H^{n-1}$-null. Combining the area formula with Lemma~\ref{l3} and then integrating in $a$ yields, by the coarea formula,
        \begin{equation}\label{eps-est}
        (M-\eta)\mathcal H^{n-1}(\partial B_1)
        \leq
        \int_{E_{\epsilon,\eta}}
        \left(\frac{f_\epsilon^-}{\lambda(n-1)}\right)^{n-1}dx,
        \end{equation}
        where
        \[
        E_{\epsilon,\eta}
        :=\{\Gamma^\epsilon=u^\epsilon\}\cap\{u^\epsilon>\eta\}.
        \]

        Let $\epsilon_j\downarrow0$ be a sequence along which the right-hand side of \eqref{eps-est} converges to its limsup. By Lemma~\ref{lem4},
        \[
        E_{\epsilon_j,\eta}
        \subseteq\Theta(u^{\epsilon_j}).
        \]
        Properties (v) and (vi) of Proposition~\ref{p1}, together with Lemma~\ref{lem3}, yield
        \[
        \limsup_{j\to\infty}E_{\epsilon_j,\eta}
        \subseteq\Theta(u)\cap\{u\geq\eta\}.
        \]
        Here the additional condition $u\geq\eta$ follows from property (v), since $u^{\epsilon_j}>\eta$ on $E_{\epsilon_j,\eta}$.

        By Proposition~\ref{p1}(iii), the sets $E_{\epsilon_j,\eta}$ lie in one fixed compact set $K_0\subset\R^n$. Enlarging $K_0$ if necessary, we may choose a compact set $K\subset\R^n$ such that $B_{c_{\epsilon_j}}(x)\subset K$ for every $x\in K_0$ and all sufficiently large $j$. Since $f$ is uniformly continuous on $K$ and $c_{\epsilon_j}\to0$, the definition of $f_{\epsilon_j}$ gives
        \[
        f_{\epsilon_j}\longrightarrow f
        \qquad\text{uniformly on }K_0.
        \]
        Hence the functions
        \[
        \mathbf 1_{E_{\epsilon_j,\eta}}(f_{\epsilon_j}^-)^{n-1}
        \]
        are uniformly bounded by an integrable function supported in $K_0$, and
        \[
        \limsup_{j\to\infty}
        \mathbf 1_{E_{\epsilon_j,\eta}}(x)(f_{\epsilon_j}^-(x))^{n-1}
        \leq
        \mathbf 1_{\limsup_{j\to\infty}E_{\epsilon_j,\eta}}(x)(f^-(x))^{n-1}.
        \]
        The reverse Fatou lemma and the preceding inclusion therefore yield
        \[
        \limsup_{j\to\infty}
        \int_{E_{\epsilon_j,\eta}}(f_{\epsilon_j}^-)^{n-1}\,dx
        \leq
        \int_{\Theta(u)\cap\{u\geq\eta\}}(f^-)^{n-1}\,dx.
        \]
        Letting $j\to\infty$ in \eqref{eps-est}, and then $\eta\downarrow0$, we use
        \[
        \{u\geq\eta\}\uparrow\{u>0\}
        \]
        to obtain the desired estimate by monotone convergence.
    \end{proof}

    \begin{proof}[Proof of Theorem~\ref{thm1}]
        If $\sup_\Omega u\leq0$, the assertion is immediate. Otherwise, apply the reduction in Remark~\ref{re1} and continue to denote the resulting global functions by $u$, $F$, and $f$. Lemma~\ref{lem5} gives
        \[
        \sup_\Omega u=\sup_{\R^n}u
        \leq
        \frac{1}{\lambda^{n-1}(n-1)^{n-1}\mathcal H^{n-1}(\partial B_1)}
        \int_{\Theta(u)\cap\{u>0\}}(f^-)^{n-1}\,dx.
        \]
        By Remark~\ref{re1}, $\Theta(u)\cap\Omega=\Theta(u^+,\Omega)$, while $\{u>0\}\subset\Omega$ and the extensions of $F$ and $f$ agree with the original functions on $\overline\Omega$. Therefore the last integral is bounded above by
        \[
        \int_{\Theta(u^+,\Omega)}(f^-)^{n-1}\,dx,
        \]
        where the right-hand side is understood in terms of the original $f$ on $\Omega$. This proves \eqref{ABP}.
    \end{proof}

\section*{Declaration of generative AI and AI-assisted technologies}
During the preparation of this manuscript, the author used ChatGPT (OpenAI) to assist with language editing, organization, and the checking of the exposition of mathematical arguments and references. The author independently verified and revised all mathematical arguments, references, and the final text, and takes full responsibility for the content of the manuscript.

\bibliographystyle{plain}
\bibliography{ABP_weighted_1Laplacian}

@article{alek1961,
author = {Aleksandrov, A. D.},
title = {Certain estimates for the {D}irichlet problem},
journal = {Dokl. Akad. Nauk SSSR},
volume = {134},
year = {1960},
pages = {1001--1004},
note = {English translation: Soviet Math. Dokl. 1 (1961), 1151--1154}
}

@article {alek1963,
    AUTHOR = {Aleksandrov, A. D.},
     TITLE = {Uniqueness conditions and bounds for the solution of the
              {D}irichlet problem},
   JOURNAL = {Vestnik Leningrad. Univ. Ser. Mat. Meh. Astronom.},
  FJOURNAL = {Vestnik Leningrad. Univ. Ser. Mat. Meh. Astronom.},
    VOLUME = {18},
      YEAR = {1963},
    NUMBER = {3},
     PAGES = {5--29},
   MRCLASS = {35.45 (53.75)},
  MRNUMBER = {164135},
MRREVIEWER = {H.\ Busemann},
}

@article {ACP,
    AUTHOR = {Argiolas, R. and Charro, F. and Peral, I.},
     TITLE = {On the {A}leksandrov-{B}akel'man-{P}ucci estimate for some
              elliptic and parabolic nonlinear operators},
   JOURNAL = {Arch. Ration. Mech. Anal.},
  FJOURNAL = {Archive for Rational Mechanics and Analysis},
    VOLUME = {202},
      YEAR = {2011},
    NUMBER = {3},
     PAGES = {875--917},
      ISSN = {0003-9527,1432-0673},
   MRCLASS = {35K55 (35B45 35K65)},
  MRNUMBER = {2854672},
       DOI = {10.1007/s00205-011-0434-y},
       URL = {https://doi.org/10.1007/s00205-011-0434-y},
}

@article{bak1961,
author = {Bakelman, I. Ja.},
title = {On the theory of quasilinear elliptic equations},
journal = {Sibirsk. Mat. Zh.},
volume = {2},
year = {1961},
pages = {179--186},
issn = {0037-4474},
mrclass = {35.47},
mrnumber = {126604},
mrreviewer = {H. Busemann}
}

@article {BGJ,
    AUTHOR = {Barron, E. N. and Goebel, R. and Jensen, R. R.},
     TITLE = {Quasiconvex functions and nonlinear {PDE}s},
   JOURNAL = {Trans. Amer. Math. Soc.},
  FJOURNAL = {Transactions of the American Mathematical Society},
    VOLUME = {365},
      YEAR = {2013},
    NUMBER = {8},
     PAGES = {4229--4255},
      ISSN = {0002-9947,1088-6850},
   MRCLASS = {35J60 (35B51 35D40 52A41 53A10)},
  MRNUMBER = {3055695},
MRREVIEWER = {Barbara\ Brandolini},
       DOI = {10.1090/S0002-9947-2013-05760-1},
       URL = {https://doi.org/10.1090/S0002-9947-2013-05760-1},
}

@book {CafCab,
    AUTHOR = {Caffarelli, L. A. and Cabr\'e, X.},
     TITLE = {Fully nonlinear elliptic equations},
    SERIES = {American Mathematical Society Colloquium Publications},
    VOLUME = {43},
 PUBLISHER = {American Mathematical Society, Providence, RI},
      YEAR = {1995},
     PAGES = {vi+104},
      ISBN = {0-8218-0437-5},
   MRCLASS = {35J60 (35-01 35B45 35B65 35Dxx)},
  MRNUMBER = {1351007},
MRREVIEWER = {P.\ Lindqvist},
       DOI = {10.1090/coll/043},
       URL = {https://doi.org/10.1090/coll/043},
}

@article {CCKS,
    AUTHOR = {Caffarelli, L. and Crandall, M. G. and Kocan, M. and Swi\c{e}ch, A.},
     TITLE = {On viscosity solutions of fully nonlinear equations with
              measurable ingredients},
   JOURNAL = {Comm. Pure Appl. Math.},
  FJOURNAL = {Communications on Pure and Applied Mathematics},
    VOLUME = {49},
      YEAR = {1996},
    NUMBER = {4},
     PAGES = {365--397},
      ISSN = {0010-3640,1097-0312},
   MRCLASS = {35J60 (35D05 35D10)},
  MRNUMBER = {1376656},
MRREVIEWER = {Katsuyuki\ Ishii},
       DOI = {10.1002/(sici)1097-0312(199604)49:4<365::aid-cpa3>3.0.co;2-a},
       URL =
              {https://doi.org/10.1002/(sici)1097-0312(199604)49:4<365::aid-cpa3>3.0.co;2-a},
}

@article {CDDM,
    AUTHOR = {Charro, F. and De Philippis, G. and Di Castro, A.
              and M\'aximo, D.},
     TITLE = {On the {A}leksandrov-{B}akelman-{P}ucci estimate for the
              infinity {L}aplacian},
   JOURNAL = {Calc. Var. Partial Differential Equations},
  FJOURNAL = {Calculus of Variations and Partial Differential Equations},
    VOLUME = {48},
      YEAR = {2013},
    NUMBER = {3-4},
     PAGES = {667--693},
      ISSN = {0944-2669,1432-0835},
   MRCLASS = {35J60 (35B45)},
  MRNUMBER = {3116027},
       DOI = {10.1007/s00526-012-0567-3},
       URL = {https://doi.org/10.1007/s00526-012-0567-3},
}

@article {CIL,
    AUTHOR = {Crandall, M. G. and Ishii, H. and Lions,
              P.-L.},
     TITLE = {User's guide to viscosity solutions of second order partial
              differential equations},
   JOURNAL = {Bull. Amer. Math. Soc. (N.S.)},
  FJOURNAL = {American Mathematical Society. Bulletin. New Series},
    VOLUME = {27},
      YEAR = {1992},
    NUMBER = {1},
     PAGES = {1--67},
      ISSN = {0273-0979,1088-9485},
   MRCLASS = {35J60 (35B05 35D05 35G20)},
  MRNUMBER = {1118699},
MRREVIEWER = {P.\ Szeptycki},
       DOI = {10.1090/S0273-0979-1992-00266-5},
       URL = {https://doi.org/10.1090/S0273-0979-1992-00266-5},
}

@article {DFQ09,
    AUTHOR = {D\'avila, G. and Felmer, P. and Quaas, A.},
     TITLE = {Alexandroff-{B}akelman-{P}ucci estimate for singular or
              degenerate fully nonlinear elliptic equations},
   JOURNAL = {C. R. Math. Acad. Sci. Paris},
  FJOURNAL = {Comptes Rendus Math\'ematique. Acad\'emie des Sciences. Paris},
    VOLUME = {347},
      YEAR = {2009},
    NUMBER = {19-20},
     PAGES = {1165--1168},
      ISSN = {1631-073X,1778-3569},
   MRCLASS = {35J60},
  MRNUMBER = {2566996},
       DOI = {10.1016/j.crma.2009.09.009},
       URL = {https://doi.org/10.1016/j.crma.2009.09.009},
}

@book {MR257325,
    AUTHOR = {Federer, H.},
     TITLE = {Geometric measure theory},
    SERIES = {Die Grundlehren der mathematischen Wissenschaften},
    VOLUME = {Band 153},
 PUBLISHER = {Springer-Verlag New York, Inc., New York},
      YEAR = {1969},
     PAGES = {xiv+676},
   MRCLASS = {28.80 (26.00)},
  MRNUMBER = {257325},
MRREVIEWER = {J.\ E.\ Brothers},
}

@article {I11,
    AUTHOR = {Imbert, C.},
     TITLE = {Alexandroff-{B}akelman-{P}ucci estimate and {H}arnack
              inequality for degenerate/singular fully non-linear elliptic
              equations},
   JOURNAL = {J. Differential Equations},
  FJOURNAL = {Journal of Differential Equations},
    VOLUME = {250},
      YEAR = {2011},
    NUMBER = {3},
     PAGES = {1553--1574},
      ISSN = {0022-0396,1090-2732},
   MRCLASS = {35J60 (35B45 35B65 35D40 35J70 35J75 49L25)},
  MRNUMBER = {2737217},
MRREVIEWER = {Fabiana\ Leoni},
       DOI = {10.1016/j.jde.2010.07.005},
       URL = {https://doi.org/10.1016/j.jde.2010.07.005},
}

@article {KS79,
    AUTHOR = {Krylov, N. V. and Safonov, M. V.},
     TITLE = {An estimate for the probability of a diffusion process hitting
              a set of positive measure},
   JOURNAL = {Dokl. Akad. Nauk SSSR},
  FJOURNAL = {Doklady Akademii Nauk SSSR},
    VOLUME = {245},
      YEAR = {1979},
    NUMBER = {1},
     PAGES = {18--20},
      ISSN = {0002-3264},
   MRCLASS = {60J60},
  MRNUMBER = {525227},
MRREVIEWER = {D.\ A.\ Darling},
}

@article {Saf80,
    AUTHOR = {Safonov, M. V.},
     TITLE = {The {H}arnack inequality for elliptic equations and {H}{\"o}lder
              continuity of their solutions},
   JOURNAL = {Zap. Nauchn. Sem. Leningrad. Otdel. Mat. Inst. Steklov. (LOMI)},
    VOLUME = {96},
      YEAR = {1980},
     PAGES = {272--287},
      NOTE = {English translation: J. Soviet Math. 21 (1983), no. 5, 851--863},
       DOI = {10.1007/BF01094448},
       URL = {https://doi.org/10.1007/BF01094448},
}

@article {LP20,
    AUTHOR = {Lewicka, M. and Peres, Y.},
     TITLE = {Which domains have two-sided supporting unit spheres at every
              boundary point?},
   JOURNAL = {Expo. Math.},
  FJOURNAL = {Expositiones Mathematicae},
    VOLUME = {38},
      YEAR = {2020},
    NUMBER = {4},
     PAGES = {548--558},
      ISSN = {0723-0869,1878-0792},
   MRCLASS = {51F30 (35F30 46B20)},
  MRNUMBER = {4177956},
MRREVIEWER = {Gabriel\ Pallier},
       DOI = {10.1016/j.exmath.2019.01.003},
       URL = {https://doi.org/10.1016/j.exmath.2019.01.003},
}

@article {puc1966,
    AUTHOR = {Pucci, C.},
     TITLE = {Limitazioni per soluzioni di equazioni ellittiche},
   JOURNAL = {Ann. Mat. Pura Appl. (4)},
  FJOURNAL = {Annali di Matematica Pura ed Applicata. Serie Quarta},
    VOLUME = {74},
      YEAR = {1966},
     PAGES = {15--30},
      ISSN = {0003-4622},
   MRCLASS = {35.19},
  MRNUMBER = {214905},
MRREVIEWER = {G.\ C.\ Barozzi},
       DOI = {10.1007/BF02416445},
       URL = {https://doi.org/10.1007/BF02416445},
}

@article {Miotto10,
    AUTHOR = {Junges Miotto, T.},
     TITLE = {The {A}leksandrov--{B}akelman--{P}ucci estimates for singular
              fully nonlinear operators},
   JOURNAL = {Commun. Contemp. Math.},
  FJOURNAL = {Communications in Contemporary Mathematics},
    VOLUME = {12},
      YEAR = {2010},
    NUMBER = {4},
     PAGES = {607--627},
       DOI = {10.1142/S0219199710003956},
       URL = {https://doi.org/10.1142/S0219199710003956},
}

@article {WW13,
    AUTHOR = {Wang, Tingting and Wang, Lizhou},
     TITLE = {A new {A}leksandrov--{B}akelman--{P}ucci maximum principle for
              {$p$}-{L}aplacian operator},
   JOURNAL = {Nonlinear Anal.},
  FJOURNAL = {Nonlinear Analysis: Theory, Methods \& Applications},
    VOLUME = {77},
      YEAR = {2013},
     PAGES = {171--179},
       DOI = {10.1016/j.na.2012.09.014},
       URL = {https://doi.org/10.1016/j.na.2012.09.014},
}

@article {KawSch07,
    AUTHOR = {Kawohl, Bernd and Schuricht, Friedemann},
     TITLE = {Dirichlet problems for the {$1$}-{L}aplace operator, including the eigenvalue problem},
   JOURNAL = {Commun. Contemp. Math.},
  FJOURNAL = {Communications in Contemporary Mathematics},
    VOLUME = {9},
      YEAR = {2007},
    NUMBER = {4},
     PAGES = {515--543},
       DOI = {10.1142/S0219199707002514},
       URL = {https://doi.org/10.1142/S0219199707002514},
}

@article {Anz83,
    AUTHOR = {Anzellotti, Gabriele},
     TITLE = {Pairings between measures and bounded functions and compensated compactness},
   JOURNAL = {Ann. Mat. Pura Appl. (4)},
  FJOURNAL = {Annali di Matematica Pura ed Applicata. Serie Quarta},
    VOLUME = {135},
      YEAR = {1983},
     PAGES = {293--318},
       DOI = {10.1007/BF01781073},
       URL = {https://doi.org/10.1007/BF01781073},
}

@book {ACM04,
    AUTHOR = {Andreu-Vaillo, Fuensanta and Caselles, Vicent and Maz\'on, Jos\'e M.},
     TITLE = {Parabolic quasilinear equations minimizing linear growth functionals},
    SERIES = {Progress in Mathematics},
    VOLUME = {223},
 PUBLISHER = {Birkh\"auser Verlag, Basel},
      YEAR = {2004},
     PAGES = {xiv+340},
       DOI = {10.1007/978-3-0348-7928-6},
       URL = {https://doi.org/10.1007/978-3-0348-7928-6},
}

@article {Dem04,
    AUTHOR = {Demengel, Fran\c{c}oise},
     TITLE = {Functions locally almost {$1$}-harmonic},
   JOURNAL = {Appl. Anal.},
  FJOURNAL = {Applicable Analysis},
    VOLUME = {83},
      YEAR = {2004},
    NUMBER = {9},
     PAGES = {865--896},
       DOI = {10.1080/00036810310001621369},
       URL = {https://doi.org/10.1080/00036810310001621369},
}

@article{ColSal03,
  author  = {Colesanti, Andrea and Salani, Paolo},
  title   = {Quasi-concave envelope of a function and convexity of level sets of solutions to elliptic equations},
  journal = {Math. Nachr.},
  volume  = {258},
  number  = {1},
  pages   = {3--15},
  year    = {2003},
  doi     = {10.1002/mana.200310083}
}

\end{document}